\documentclass[11pt]{article}
\usepackage{amsfonts}
\usepackage{amssymb}
\usepackage{hyperref}
\usepackage{pstricks}
\usepackage{pst-plot}
\usepackage{pst-node}
\usepackage{multido}

\newtheorem{definition}{Definition}
\newtheorem{example}{Example}
\newtheorem{notation}{Notation}
\newtheorem{theorem}{Theorem}
\newtheorem{proposition}{Proposition}
\newtheorem{corollary}{Corollary}
\newtheorem{claim}{Claim}
\newtheorem{remark}{Remark}
\newcommand{\qed}{\rule{7pt}{7pt}}
\newenvironment{proof}{\noindent{\bf Proof}\hspace*{1em}}{\qed\bigskip}

\newcommand{\Var}{{\rm Var}}

\newcommand{\PutGrid}[3]{\
\multido{\rA=#1+#3}{7}{ \ 
   \multido{\rB=#2+#3}{7}{ \ 
      \psline(\rA,#2)(\rA,\rB) \
    } \
} \ \multido{\rA=#2+#3}{7}{ \
   \multido{\rB=#1+#3}{7}{ \
      \psline(#1,\rA)(\rB,\rA) \
  }\
 } \
}

\begin{document}

\title{On The Influences of Variables on Boolean Functions in Product Spaces}

\author{
Nathan Keller \\
Einstein Institute of Mathematics, Hebrew University\\
Jerusalem 91904, Israel\\
{\tt nkeller@math.huji.ac.il}\\
}

\maketitle

\begin{abstract}
In this paper we consider the influences of variables on Boolean
functions in general product spaces. Unlike the case of functions
on the discrete cube where there is a clear definition of
influence, in the general case at least three definitions were
presented in different papers. We propose a family of definitions
for the influence, that contains all the known definitions, as
well as other natural definitions, as special cases. We prove a generalization 
of the BKKKL theorem, which is tight in terms of the
definition of influence used in the assertion, and use it to
generalize several known results on influences in general product
spaces.
\end{abstract}

\section{Introduction}

Influences of variables on Boolean functions have been extensively
studied during the last few decades. This study led to important
applications in Theoretical Computer Science, Combinatorics,
Mathematical Physics, Social Choice Theory, and other areas. The
basic results on influences were obtained for functions on the
discrete cube, but some applications required a generalization of
the results to more general product spaces. Unlike the discrete
case, where there exists a single natural definition of influence,
for general product spaces at least three definitions were
presented in different papers. All the definitions are based on
dividing the space into subspaces called {\it fibers}.
\begin{notation}
In the following definitions, $X$ denotes a product space 
$X = X_1 \times X_2 \times \ldots \times X_n$, endowed with a product
measure $\mu = \mu_1 \otimes \mu_2 \otimes \ldots \otimes \mu_n$. 
Throughout the paper, $\log n$ denotes $\log_2 n$.
\end{notation}
\begin{definition}
For $x=(x_1,\ldots,x_n) \in X$ and for $1 \leq k \leq n$, the fiber
of $x$ in the $k$-th direction is
\[
s_k(x)=\{y:y_i=x_i, \forall i \neq k\}.
\]
\end{definition}
The influence of the $k$-th variable on a function $f:X
\rightarrow \{0,1\}$ is the expectation, over all the fibers in
the $k$-th direction, of the influence of the variable on each
fiber. The original definition of influences in product spaces,
introduced in~\cite{BKKKL}, is the following:
\begin{definition}
For $f:X \rightarrow \{0,1\}$, and for $1\leq k\leq n$, the
influence of the $k$-th variable on $f$ is
\[
I_f(k)=\mu \Big[\{x\in X: f \mbox{ is non-constant on } s_k(x)\}
\Big].
\]
\label{Def:BKKKL}
\end{definition}
Another widely used definition (see, e.g.,
~\cite{Hatami,Keller,Maj-Stablest}), is:
\begin{definition}
For a function $f:X \rightarrow \{0,1\}$, an element
$x=(x_1,\ldots,x_n)\in X$, and for $1\leq k\leq n$, denote the
restriction of $f$ to the fiber $s_k(x)$ by $f_k^x:X_k \rightarrow
\{0,1\}$. That is, for all $t\in X_k$,
\[
f_k^x(t)=f(x_1,\ldots,x_{k-1},t,x_{k+1},\ldots,x_n).
\]
The influence of the $k$-th variable on $f$ is
\[
\tilde{I}_f(k)=\mathbb{E}_{x \in X} \Var(f_k^x).
\]
\label{Def:Tilde}
\end{definition}
The difference between the definitions is demonstrated by the
following example:
\begin{example}
Consider $X=[0,1]^n$, endowed with the Lebesgue measure
$\lambda$. Let $f:[0,1]^n \rightarrow \{0,1\}$ be defined by:
\[
f(x)=1 \Longleftrightarrow x_i>1/n, \forall 1\leq i\leq n.
\]
It is easy to see that for all $1 \leq k \leq n$ we have $I_k(f)
\approx 1/e$, while $\tilde{I}_k(f) \approx (n-1)/en^2$.
\end{example}
In general, for every function $f$ we have $\tilde{I}_k(f) \leq
I_k(f)/4$, but in many cases $\tilde{I}_k(f)$ is much smaller than
$I_k(f)$.

\medskip

\noindent The most well-known theorem concerning influences in
general product spaces is the BKKKL theorem~\cite{BKKKL}:
\begin{theorem}[Bourgain, Kahn, Kalai, Katznelson, and Linial]
For every $f:X \rightarrow \{0,1\}$ such that $\mathbb{E}f=p$,
there exists a variable $k$ such that $I_f(k) \geq cp(1-p)\log
n/n$, where $c$ is a universal constant.
\label{Theorem:BKKKL}
\end{theorem}
The BKKKL theorem was used to obtain several important results.
For example, it was a
central tool in showing that any monotone graph property has a
sharp threshold~\cite{Friedgut-Kalai}.

\medskip

\noindent In this paper we present a generalized definition of
influences in general product spaces, which contains
Definitions~\ref{Def:BKKKL} and~\ref{Def:Tilde} as special cases:
\begin{definition}
Let $h:[0,1] \rightarrow [0,1]$. For every $f:X \rightarrow
\{0,1\}$ and for each $1\leq k\leq n$, the \textbf{h-influence} of
the $k$-th coordinate on $f$ is
\[
I_f^h(k) = \mathbb{E}_{x \in X} h(\mathbb{E}f_k^x).
\]
\label{Def:h-influence}
\end{definition}
Definition~\ref{Def:BKKKL} is obtained by substituting $h(t)=1$ if
$t \neq 0,1$, and $h(t)=0$ otherwise.\footnote{Here we assume that
a function is constant on a fiber if it is constant except for a
set of measure zero.} Definition~\ref{Def:Tilde} is obtained by
substituting $h(t)=t(1-t)$.\footnote{We note that influences
toward zero and one (see, e.g.,~\cite{Friedgut2}), can be also
expressed as $h$-influences. Influence toward one is obtained by
$h(t)=1-t$ for $t \neq 0$, and $h(0)=0$. Influence toward zero is
obtained by $h(t)=t$ for $t \neq 1$, and $h(1)=0$.}

\medskip

\noindent We show that the BKKKL theorem can be generalized by
replacing the influence $I_k(f)$ with a smaller $h$-influence, and
obtain a full characterization of the $h$-influences for which the
BKKKL theorem holds. As in~\cite{BKKKL}, we prove our
results in the case $X=[0,1]^n$, endowed with the Lebesgue
measure $\lambda$. This case is considered fairly general
(see~\cite{BKKKL,Hatami}), since many cases of interest can be easily
reduced to it.
\begin{theorem}
Denote the Entropy function $H(t)=-t\log t -(1-t)\log(1-t)$ by
$Ent(t)$. Let $h:[0,1] \rightarrow [0,1]$ be a concave
function, such that $h(t) \geq Ent(t)$ for all $0 \leq t \leq 1$.
Then for every $f:[0,1]^n \rightarrow \{0,1\}$ with
$\mathbb{E}f=p$, there exists $1\leq k \leq n$ such that the
h-influence of the $k$-th variable on $f$ satisfies
\[
I_f^h(k) \geq cp(1-p)\log n/n,
\]
where $c$ is a universal constant.
\label{Our-Main-Theorem}
\end{theorem}
The advantage of Theorem~\ref{Our-Main-Theorem} over the BKKKL
theorem is demonstrated by the function $f(x)=1
\Longleftrightarrow (x_i>1/n, \forall 1\leq i\leq n)$, presented
above. For this function we have $\mathbb{E}(f) \approx 1/e$, and
for all $k$, $I_k(f) \approx 1/e$, and hence the BKKKL theorem is
far from being tight in this case. On the other hand, for
$h(t)=Ent(t)$ we get $I_k(f) \approx Ent(1/n)/e \approx c' \log
n/n$, and thus Theorem~\ref{Our-Main-Theorem} is tight in this
case.

\medskip

\noindent Furthermore, by examining variants of the {\it tribes}
function presented in~\cite{BL}, we show that
Theorem~\ref{Our-Main-Theorem} is tight in the following sense:
\begin{proposition}
Let $h:[0,1] \rightarrow [0,1]$ and let $\epsilon>0$. If there
exists $0<q<1$ such that $h(q) \leq \epsilon Ent(q)$, then there
exists a function $f:[0,1]^n \rightarrow \{0,1\}$ such that
$\mathbb{E}f=\Theta(1)$, and for all $1\leq k \leq n$, the
h-influence of the $k$-th variable on $f$ satisfies $I_f^h(k) \leq
c \epsilon \log n/n$, where $c$ is a universal constant.
\end{proposition}
The proof of Theorem~\ref{Our-Main-Theorem} for monotone functions
is simple, and essentially follows the proof of the BKKKL theorem.
The reduction from general functions to monotone functions is much
more complicated in our case, and involves discretization and
convexity arguments. We note that in the special case of functions
on the discrete cube endowed with a product measure
$\mu=\mu_q^{\otimes n}$, statements similar to
Theorem~\ref{Our-Main-Theorem} were proved by Friedgut and
Kalai~(\cite{Friedgut-Kalai}, Theorem~3.1), and independently by
Talagrand~\cite{Talagrand1}.

\medskip

\noindent Using the techniques we develop in the proof of Theorem~\ref{Our-Main-Theorem}, 
we provide generalizations of several known results on influences in product
spaces. In particular, we generalize a lower bound on the vector
of influences obtained by Talagrand~\cite{Talagrand1} for
functions on the discrete cube endowed with the product measure
$\mu_q^{\otimes n}$. We obtain the following:
\begin{proposition}
Let $h:[0,1] \rightarrow [0,1]$ be a concave function
satisfying $h(t) \geq Ent(t)$ for all $0 \leq t \leq 1$. There
exists a constant $K>0$ such that for any function $f:[0,1]^n
\rightarrow \{0,1\}$ with $\mathbb{E}f=p$,
\[
p(1-p) \leq K \sum_{i \leq n}
\frac{I^h_f(i)}{\log\frac{4}{3I^h_f(i)}}.
\]
\end{proposition}
Our proof uses the proof of Talagrand's result for the case
$q=1/2$, along with the monotonization and discretization
technique for $h$-influences presented in our paper. Since
Talagrand's result (for a general $q$) follows easily from our
generalization, our technique can replace the
major part of Talagrand's proof (containing a proof of a biased
version of Beckner's hypercontractive inequality~\cite{Beckner}).

\medskip

\noindent This paper is organized as follows: In
Section~\ref{sec:lemmas} we present the monotonization and
discretization techniques for $h$-influences, which are the main
tools used in the paper. The proof of
Theorem~\ref{Our-Main-Theorem} and its applications are presented
in Section~\ref{sec:applications}. In Section~\ref{sec:Tightness} 
we discuss the tightness of our results.

\section{Monotonization and Discretization for $h$-Influences}
\label{sec:lemmas}

In this section we generalize to $h$-influences two central steps
of the proof of the BKKKL theorem. The first step is {\it
monotonization}, a shifting technique allowing to replace a
Boolean function on the continuous cube by a monotone function
with the same expectation and non-higher influences. While for
Definition~\ref{Def:BKKKL} of the influences, the shifting
argument is standard and easy, for $h$-influences the proof is
much more complicated. We present this proof, involving
discretization and a delicate argument exploiting the concavity of
the function $h$, in Section~\ref{sec:sub:monotonization}. The
second step is {\it discretization}, allowing to replace a
monotone function $f:[0,1]^n \rightarrow \{0,1\}$ by a function
$g:\{0,1\}^m \rightarrow \{0,1\}$, where $m=\Theta(n \log n)$,
having the following property: The $m$ coordinates can be divided
into $n$ equal sets of size $\Theta(\log n)$ each, such that the sum of
influences of the coordinates in each set on $g$ is not much
higher than the influence of the corresponding coordinate on $f$.
We show that the same can be performed with $h$-influences, given
that for all $t$, we have $h(t) \geq Ent(t)$. The proof of this
claim, presented in Section~\ref{sec:sub:discretization}, is quite
straightforward, and is essentially the same as the proof of Claim
2.8 in~\cite{Friedgut2}.

\subsection{Monotonization for $h$-Influences}
\label{sec:sub:monotonization}

\begin{definition}
A function $f:[0,1]^n \rightarrow \mathbb{R}$ is monotone if for
all $x,y \in [0,1]^n$,
\[
\forall i (x_i \geq y_i) \Longrightarrow f(x) \geq f(y).
\]
\end{definition}

\noindent An important component in the proofs of the KKL
theorem~\cite{KKL}, the BKKKL theorem~\cite{BKKKL} and many other
results is {\it monotonization}. Using the monotonization
technique, which is actually a standard shifting argument
(see~\cite{Frankl}), the examined function is replaced by a
monotone function with the same expectation and non-higher
influences. The rest of the proof is performed for the monotone
function, and the resulting lower bound on its influences yields a
lower bound on the influences of the original function.

\noindent In this section we prove that the monotonization
technique is valid also for $h$-influences, if $h$ is a concave
function. The condition imposed on $h$ holds for all
the ``natural'' definitions of influences in the continuous case,
except for Definition~\ref{Def:BKKKL}. However, the validity of
the monotonization for this definition was already proved
in~(\cite{BKKKL}, Lemma 1).
\begin{theorem}
Let $h:[0,1] \rightarrow [0,1]$ be a concave function.
For every Borel measurable function $f:[0,1]^n \rightarrow
\{0,1\}$, there exists a monotone function $g:[0,1]^n \rightarrow
\{0,1\}$, such that:
\begin{enumerate}
\item $\mathbb{E}f=\mathbb{E}g$.

\item For all $1\leq k \leq n$, we have $I^h_f(k) \geq I^h_g(k)$.
\end{enumerate}
\label{Thm:Monotonization}
\end{theorem}

\begin{proof}
\noindent The proof of Theorem~\ref{Thm:Monotonization} is
composed of three steps.

\subsubsection{Reduction to Functions on the Unit Square}

The procedure of the monotonization, that allows to arrive to $g$
given $f$, is the same standard shifting procedure as presented
in~\cite{BKKKL}. The monotonization is performed coordinate-wise.
For the $k$-th coordinate, we consider the fibers in the $k$-th
direction, and for each fiber $s_k(x)$, we replace the function
$f_k^x$ by the function
\[
Mf_k^x (t) = \left\lbrace
  \begin{array}{c l}
    1, & t>1-\mathbb{E}(f_k^x)\\
    0, & otherwise.
  \end{array}
\right.
\]
Clearly, the resulting function $Mf$ satisfies
$\mathbb{E}(Mf)=\mathbb{E}(f)$. Moreover, it is clear that
performing this monotonization procedure in all the coordinates
sequentially leads to a monotone function. Hence, the only claim
we have to prove is that $I^h_f(j) \geq I^h_{Mf}(j)$ for all
$1\leq j \leq n$, where $Mf$ denotes the monotonization of $f$ in
the $k$-th direction.

\medskip

\noindent For $j=k$ we have $I^h_f(j)=I^h_{Mf}(j)$, since the
expectation of $f_k^x$ on each fiber is preserved. Hence, we
assume that $j \neq k$. It is sufficient to show that for every
fixed value of the $n-2$ remaining coordinates, the contribution
of the family of fibers corresponding to this fixed value to
$I^h_f(j)$ is not less than the respective contribution to
$I^h_{Mf}(j)$. Therefore, without loss of generality we can fix
the remaining coordinates and thus assume that $n=2$, i.e.,
$f:[0,1]^2 \rightarrow \{0,1\}$.

\subsubsection{Proof of a Discrete Version of
Theorem~\ref{Thm:Monotonization}}

\noindent Let $r$ be an integer. We assume that the value of $f$
is constant on squares of the form $(\frac{l}{r},\frac{l+1}{r})
\times (\frac{m}{r},\frac{m+1}{r})$, for all $0 \leq l,m < r$.
Hence, $f$ can be represented by an $r \times r$ matrix $A$, where
$A[i,j]$ denotes the value of $f(x,y)$ for all $(x,y) \in
(\frac{i-1}{r},\frac{i}{r}) \times (\frac{j-1}{r},\frac{j}{r})$.
The monotonization (assuming, without loss of generality, that it
is performed according to the rows of the matrix) consists of
moving all the $1$-s in each row to the rightmost positions in the
same row, as described in Figure~\ref{Figure1}. 

\begin{figure}[htb]
\begin{center}
   \begin{pspicture}(0,-0.1)(10,3.5)

\PutGrid{6.5}{0.2}{0.5} \PutGrid{0.5}{0.2}{0.5}

\uput[r](0.2,-0.1){The Original Matrix} \uput[r](6.2,-0.1){The
Shifted Matrix}

\cnode(4.0,1.7){0.01}{B0} \cnode(6.0,1.7){0.01}{B1}
\ncline[arrows=->]{B0}{B1}




\uput[r](0.45,0.45){1} \uput[r](0.95,0.45){0}
\uput[r](1.45,0.45){1} \uput[r](1.95,0.45){1}
\uput[r](2.45,0.45){0} \uput[r](2.95,0.45){1}

\uput[r](0.45,0.95){0} \uput[r](0.95,0.95){1}
\uput[r](1.45,0.95){1} \uput[r](1.95,0.95){0}
\uput[r](2.45,0.95){0} \uput[r](2.95,0.95){0}

\uput[r](0.45,1.45){0} \uput[r](0.95,1.45){0}
\uput[r](1.45,1.45){0} \uput[r](1.95,1.45){1}
\uput[r](2.45,1.45){1} \uput[r](2.95,1.45){1}

\uput[r](0.45,1.95){0} \uput[r](0.95,1.95){0}
\uput[r](1.45,1.95){0} \uput[r](1.95,1.95){0}
\uput[r](2.45,1.95){1} \uput[r](2.95,1.95){0}

\uput[r](0.45,2.45){0} \uput[r](0.95,2.45){1}
\uput[r](1.45,2.45){1} \uput[r](1.95,2.45){0}
\uput[r](2.45,2.45){1} \uput[r](2.95,2.45){1}

\uput[r](0.45,2.95){1} \uput[r](0.95,2.95){0}
\uput[r](1.45,2.95){0} \uput[r](1.95,2.95){0}
\uput[r](2.45,2.95){0} \uput[r](2.95,2.95){1}


\uput[r](6.45,0.45){0} \uput[r](6.95,0.45){0}
\uput[r](7.45,0.45){1} \uput[r](7.95,0.45){1}
\uput[r](8.45,0.45){1} \uput[r](8.95,0.45){1}

\uput[r](6.45,0.95){0} \uput[r](6.95,0.95){0}
\uput[r](7.45,0.95){0} \uput[r](7.95,0.95){0}
\uput[r](8.45,0.95){1} \uput[r](8.95,0.95){1}

\uput[r](6.45,1.45){0} \uput[r](6.95,1.45){0}
\uput[r](7.45,1.45){0} \uput[r](7.95,1.45){1}
\uput[r](8.45,1.45){1} \uput[r](8.95,1.45){1}

\uput[r](6.45,1.95){0} \uput[r](6.95,1.95){0}
\uput[r](7.45,1.95){0} \uput[r](7.95,1.95){0}
\uput[r](8.45,1.95){0} \uput[r](8.95,1.95){1}

\uput[r](6.45,2.45){0} \uput[r](6.95,2.45){0}
\uput[r](7.45,2.45){1} \uput[r](7.95,2.45){1}
\uput[r](8.45,2.45){1} \uput[r](8.95,2.45){1}

\uput[r](6.45,2.95){0} \uput[r](6.95,2.95){0}
\uput[r](7.45,2.95){0} \uput[r](7.95,2.95){0}
\uput[r](8.45,2.95){1} \uput[r](8.95,2.95){1}
\end{pspicture}
\end{center}
\caption{The Matrices Representing $f$ and $Mf$}
\label{Figure1}
\end{figure}
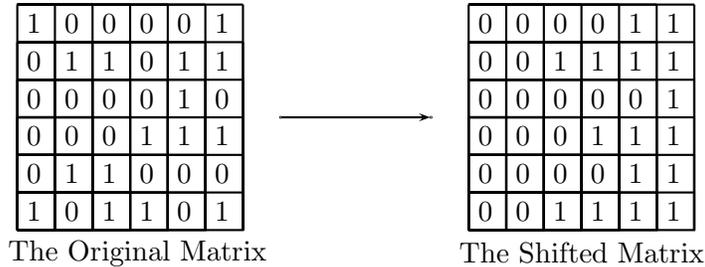

\noindent The influence in the ``columns'' direction is
\[
I_f^h(\mbox{ columns }) = \frac{1}{r} \sum_{j=1}^r h
\Big(\frac{\sum_{i=1}^r A[i,j]}{r}\Big).
\]
We show that if $h$ is concave, the influence is not increased by
our monotonization. We use two simple observations:
\begin{enumerate}
\item If two entire columns are interchanged, the influence does
not change.

\item If a `1' is moved from a column with a smaller (or equal)
number of $1$-s to a column with a bigger (or equal) number, then
the influence is not increased.
\end{enumerate}

\noindent The first observation is clear. The second observation
follows from the concavity of $h$. Indeed, denote the numbers of
1-s in the columns before the operation by $a$ and $b$, where $b
\geq a$. Since the contribution of all the untouched columns to
the influence remains unchanged, we want to prove that
\[
h(\frac{b}{r})+h(\frac{a}{r}) \geq
h(\frac{b+1}{r})+h(\frac{a-1}{r}).
\]
There exists $0<t<1$ such that $b/r=t(b+1)/r +(1-t)(a-1)/r$. By
symmetry, we have $a/r=(1-t)(b+1)/r+t(a-1)/r$. By the concavity of
$h$,
\[
h(\frac{b}{r})=h(t\frac{b+1}{r}+(1-t)\frac{a-1}{r}) \geq t
h(\frac{b+1}{r})+ (1-t) h(\frac{a-1}{r}),
\]
and similarly,
\[
h(\frac{a}{r}) \geq (1-t) h(\frac{b+1}{r})+ t h(\frac{a-1}{r}).
\]
Hence,
\[
h(\frac{a}{r})+h(\frac{b}{r}) \geq h(\frac{b+1}{r}) +
h(\frac{a-1}{r}),
\]
as asserted.

\medskip

\noindent The transformation from $f$ to $Mf$ can be treated as a concatenation
of steps of the two classes discussed in the observations above. Indeed, the
transformation can be represented by the following algorithm:

\medskip

\noindent For $i=0,1,2,\ldots,r-1$:
\begin{enumerate}
\item Consider columns $1,2,\ldots,r-i$ of the matrix.

\item Reorder the columns, such that the rightmost column will
have the maximal number of $1$-s.

\item Consider the zero values in the rightmost column (i.e., column $r-i$). For each
such value, look at the corresponding row of the matrix
(restricted to the first $r-i$ columns), and if there exist `1'
values in the row, interchange one of them with the zero value in
the rightmost column.
\end{enumerate}

\noindent For example, in Figure~\ref{Figure1}, the step of the algorithm
corresponding to $i=0$ consists of interchanging the `1' in the
cell $A[3,5]$ with the zero in the cell $A[3,6]$, and
interchanging the `1' in the cell $A[5,3]$ with the zero in the
cell $A[5,6]$. (There is no reordering of the columns since the
rightmost column already has the maximal number of $1$-s). In the
remaining steps of the algorithm, the $r$-th column is left
untouched.

\medskip

\noindent In the output of the algorithm, all the $1$-s in each
row are concentrated in the rightmost cells. Since during the
algorithm, only cells in the same row are interchanged, the
output of the algorithm is the matrix representing $Mf$.
The steps of the algorithm are indeed of the two classes discussed
in the observations above (reordering the columns and moving a `1'
from a column with a smaller (or equal) number of $1$-s to a column with
a bigger (or equal) number of $1$-s). Therefore, the $h$-influence is
not increased by the transformation from $f$ to $Mf$. This
proves the assertion of Theorem~\ref{Thm:Monotonization} for a discrete
version of the function $f$.

\subsubsection{Reduction from the General Case to the Discrete
Case}

\noindent We use the following classical result in Measure theory
(see, for example,~\cite{Bauer}, Theorem~5.7):
\begin{theorem}
Let $\mu$ be a finite measure on a $\sigma$-algebra $A$ of subsets
of a set $\Omega$, which is generated by an algebra $A_0$. Then
for each $C \in A$, there exists a sequence
$\{C_n\}_{n=1}^{\infty}$ such that $C_n \in A_0$ for all $n$, and
\[
\lim_{n \rightarrow \infty} \mu(C_n \triangle C) = 0,
\]
where $X \triangle Y = (X \setminus Y) \cup (Y \setminus X)$ is the
symmetric difference between the sets $X$ and $Y$.
\end{theorem}

\noindent Since the algebra generated by the diadic squares (i.e.,
the squares $(\frac{l}{2^s},\frac{l+1}{2^s}) \times
(\frac{m}{2^s},\frac{m+1}{2^s})$ for all $s$ and all
$0 \leq l,m < 2^s$) generates the Borel $\sigma$-algebra on the
unit square, we get immediately the following corollary:
\begin{corollary}
For every Borel measurable function $f:[0,1]^2 \rightarrow
\{0,1\}$ and every $\epsilon>0$, there exists an integer $s$ and a
function $f_s$ which is constant on squares of the form
$(\frac{l}{2^s},\frac{l+1}{2^s}) \times
(\frac{m}{2^s},\frac{m+1}{2^s})$ for all $0 \leq l,m < 2^s$, such
that
\[
\int_0^1 \int_0^1 | f - f_s | dx dy \leq \epsilon.
\]
\label{Cor3.1}
\end{corollary}

\noindent Let $f:[0,1]^2 \rightarrow \{0,1\}$ be a Borel
measurable function, and let $\epsilon>0$. We want to show that
$I^h_f(2) \geq I^h_{Mf}(2) - \epsilon$, where $Mf$ denotes the
monotonization of $f$ in the first direction (i.e., with respect
to the $x$ coordinate). We approximate $f$
by a ``discrete'' function $f_s$ and use the result on discrete
functions obtained in the previous step of the proof.

\medskip

\noindent Since $h$ is concave on $[0,1]$, and thus continuous, 
there exists $\delta$ such that
if $|x-y| \leq \delta$, then $|h(x)-h(y)| \leq \epsilon/4$. By
Corollary~\ref{Cor3.1}, there exists a function $f_s$ which is
constant on squares of the form $(\frac{l}{2^s},\frac{l+1}{2^s})
\times (\frac{m}{2^s},\frac{m+1}{2^s})$, such that
\begin{equation}
\int_0^1 \int_0^1 |f(x,y) - f_s(x,y)| dx dy \leq \delta \epsilon /4.
\label{Eq3.3.1}
\end{equation}
Let
\[
S_1=\{x \in [0,1]: \int_0^1 |f(x,y)-f_s(x,y)| dy \leq \delta\},
\]
and $S_2 = [0,1] \setminus S_1$. By the Fubini theorem and
Inequality~(\ref{Eq3.3.1}), we have $\lambda(S_2) \leq
\epsilon/4$,
where $\lambda$ is the Lebesgue measure on $\mathbb{R}$. Therefore,
\begin{equation}
\int_{S_2} |h(\int_0^1 f(x,y) dy) - h(\int_0^1 f_s(x,y) dy)| dx \leq \epsilon/4.
\label{Eq3.3.2}
\end{equation}

\medskip

\noindent On the other hand, if $x \in S_1$, then
\[
|\int_0^1 f(x,y) dy - \int_0^1 f_s(x,y) dy| \leq \int_0^1
|f(x,y)-f_s(x,y)| dy \leq \delta,
\]
and hence,
\[
|h(\int_0^1 f(x,y) dy) - h(\int_0^1 f_s(x,y) dy)| \leq \epsilon/4.
\]
Thus,
\begin{equation}
\int_{S_1} |h(\int_0^1 f(x,y) dy) - h(\int_0^1 f_s(x,y) dy)| dx \leq \epsilon/4.
\label{Eq3.3.3}
\end{equation}

\noindent Using the triangle inequality and Inequalities~(\ref{Eq3.3.2}) and~(\ref{Eq3.3.3}),
we get:
\[
|I^h_f(2) - I^h_{f_s}(2)| = |\int_0^1 h(\int_0^1 f(x,y) dy) dx - \int_0^1 h(\int_0^1 f_s(x,y) dy) dx|
\leq
\]
\[
\leq \int_0^1 |h(\int_0^1 f(x,y) dy) - h(\int_0^1 f_s(x,y) dy)| dx =
\int_{S_1} |h(\int_0^1 f(x,y) dy) - h(\int_0^1 f_s(x,y) dy)| dx +
\]
\begin{equation}
+ \int_{S_2} |h(\int_0^1 f(x,y) dy) - h(\int_0^1 f_s(x,y) dy)| dx
\leq \epsilon/4 + \epsilon/4 = \epsilon/2.
\label{Eq3.3.4}
\end{equation}

\medskip

\noindent In order to bound the term $|I^h_{Mf}(2) - I^h_{Mf_s}(2)|$, we use the following
property of the monotonization operation:
\begin{claim}
For every pair of measurable functions $f,g:[0,1]^2 \rightarrow \{0,1\}$,
\[
\int_0^1 \int_0^1 |Mf(x,y) - Mg(x,y)| dx dy \leq \int_0^1 \int_0^1 |f(x,y) - g(x,y)| dx dy,
\]
where $Mf$ and $Mg$ represent the monotonizations of $f$ and $g$ with respect to the $x$
coordinate.
\label{Claim3.1}
\end{claim}
\begin{proof}
Fix $y \in [0,1]$. Denote $t_0= 1-\int_0^1 f(x,y) dx$, and $t_1=1-
\int_0^1 g(x,y) dx$. By the definition of $Mf$ and $Mg$, we have
$Mf(x,y)=0$ if and only if $x<t_0$, and $Mg(x,y)=0$ if and only if
$x<t_1$. Assume w.l.o.g. that $t_0<t_1$. Then
$|Mf(x,y)-Mg(x,y)|=1$ for $t_0 \leq x <t_1$, and
$|Mf(x,y)-Mg(x,y)|=0$ otherwise. Therefore,
\[
\int_0^1 |Mf(x,y)-Mg(x,y)|dx = |(1-\int_0^1 f(x,y) dx) -
(1-\int_0^1 g(x,y) dx)| \leq \int_0^1 |f(x,y) - g(x,y)| dx.
\]
Hence, by the Fubini theorem,
\[
\int_0^1 \int_0^1 |Mf(x,y) - Mg(x,y)| dx dy \leq
\int_0^1 \int_0^1 |f(x,y) - g(x,y)| dx dy,
\]
as asserted.
\end{proof}

\noindent Combining Claim~\ref{Claim3.1} and Inequality~(\ref{Eq3.3.1}) we get:
\[
\int_0^1 \int_0^1 |Mf(x,y) - Mf_s(x,y)| dx dy \leq \delta \epsilon/4,
\]
and hence by the argument applied above to $|I^h_f(2) - I^h_{f_s}(2)|$,
\begin{equation}
|I^h_{Mf}(2) - I^h_{Mf_s}(2)| \leq \epsilon/2.
\label{Eq3.3.5}
\end{equation}

\noindent Finally, combining Inequalities~(\ref{Eq3.3.4})
and~(\ref{Eq3.3.5}) with the proof of the theorem in the discrete
case, we get:
\[
I^h_f(2) - I^h_{Mf}(2) = (I^h_f(2) - I^h_{f_s}(2)) + (I^h_{f_s}(2) - I^h_{Mf_s}(2)) +
(I^h_{Mf_s}(2) - I^h_{Mf}(2)) \geq
\]
\begin{equation}
\geq (-\epsilon/2) + 0 + (-\epsilon/2) = -\epsilon.
\end{equation}
This completes the proof of Theorem~\ref{Thm:Monotonization}.
\end{proof}

\subsection{Discretization for $h$-Influences}
\label{sec:sub:discretization}

\noindent Another important component in the proofs of the BKKKL
theorem (both the original proof presented in~\cite{BKKKL} and the
simpler proof presented in~\cite{Friedgut2}) is
discretization.\footnote{A sketch of another proof of the theorem,
that does not use discretization, is presented in~\cite{BKKKL}.}
In this section we follow the discretization suggested
in~\cite{Friedgut2}, and show that it can be performed not only
for Definition~\ref{Def:BKKKL} of the influences, but also for
general $h$-influences, provided that $h(t) \geq Ent(t)$ for all
$0 \leq t \leq 1$.

\medskip

\noindent The discretization suggested in~\cite{Friedgut2} is the
following. Let $l=3 \log n$. Subdivide $[0,1]^n$ into $2^{ln}$
sub-cubes by dividing each of the base intervals into $2^l$ equal
parts. It is shown in~\cite{Friedgut2} that a monotone function
$f:[0,1]^n \rightarrow \{0,1\}$ can be approximated by a function
$\tilde{f}$ that is constant on each of the small sub-cubes. This
function naturally corresponds to a function $g:\Big(\{0,1\}^{l} \Big)^n
\rightarrow \{0,1\}$ (by replacing the interval
$[m2^{-l},(m+1)2^{-l}]$ with the binary expansion of $m$). Each of
the initial variables $1 \leq i\leq n$ is now replaced by $l$
variables, $\{i_j\}_{j=1}^l$. It was shown in~\cite{BKKKL} that
for all $1 \leq i \leq n$, we have $\sum_{j=1}^l I_g(i_j) \leq 2
I_f(i)$. We prove a similar result for the $h$-influence,
given that $h$ is ``big enough''.
\begin{proposition}
Let $h:[0,1] \rightarrow [0,1]$ such that for all $0 \leq t \leq 1$, 
$h(t) \geq Ent(t)$. Let $f:[0,1]^n \rightarrow \{0,1\}$ be a monotone
function and let $g:\{0,1\}^{ln} \rightarrow \{0,1\}$ be obtained
from $f$ by the procedure described above. Then for every $1 \leq
i\leq n$,
\begin{equation}
\sum_{j=1}^l I_g(i_j) \leq 6 I_f^h(i).
\label{Eq:Discretization}
\end{equation}
\label{Prop:Discretization}
\end{proposition}
\begin{proof}
We note that it is sufficient to prove the assertion for the
contribution of any single fiber $s_i(x)$ to the influences in the
two sides of Inequality~(\ref{Eq:Discretization}). Let $x \in
[0,1]^n$, and consider the function $f_i^x$, i.e., the restriction
of $f$ to the fiber $s_i(x)$. Denote $t_0=\mathbb{E}(f_i^x)$, and
assume that $t_0 \geq 1/2$ (the case $t_0<1/2$ is treated later).
Since $f_i^x$ is monotone, we have $f_i^x(t)=0$ if $t<1-t_0$ and
$f_i^x(t)=1$ if $t>1-t_0$. By the construction of $g$, the
function $g_i^x$ satisfies $g_i^x(m)=0$ if $m<(1-t_0)2^l$ and
$g_i^x(m)=1$ if $m>(1-t_0)2^l$. Let $k$ satisfy $2^k \leq
(1-t_0)2^l < 2^{k+1}$. Then
\[
\Pr_{0 \leq m \leq 2^l -1}[g_i^x(m) \neq g_i^x(m \oplus e_j)] \leq
\left\lbrace
  \begin{array}{c l}
    2^{j+1-l}, & j \leq k\\
    2(1-t_0), & j>k.
  \end{array}
\right.
\]
Indeed, for $j \leq k$, the relation $[g_i^x(m) \neq g_i^x(m
\oplus e_j)]$ can be satisfied only for $(1-t_0)2^l-2^j \leq m
\leq (1-t_0)2^l+2^j$, and hence the probability is bounded from
above by $2^{j+1-l}$. For $j>k$, the restriction is even stricter:
$m$ must satisfy either $0 \leq m \leq (1-t_0)2^l$ or $2^j \leq m
\leq (1-t_0)2^l+2^j$, and hence the probability is bounded by (and
actually is equal to) $2(1-t_0)$. Therefore,
\begin{eqnarray*}
\sum_{j} I_{g_i^x}(i_j) \leq \sum_{j \leq k} 2^{j+1-l} +
\sum_{j>k} 2(1-t_0) \leq 2^{k+2-l} + 2(1-t_0)(l-1-k) \leq \\
\leq 4(1-t_0)+2(1-t_0)\log \frac{1}{1-t_0} \leq 6(1-t_0)\log
\frac{1}{1-t_0} \leq 6 Ent(t_0).
\end{eqnarray*}
If $t_0<1/2$, then instead of examining the function $g_i^x$, we
consider the dual function $\tilde{g}_i^x(m)=1-g_i^x(2^l-1-m)$. It
is easy to show that $\mathbb{E}(\tilde{g}_i^x)=1-t_0$, and that
for all $j$, $I_{\tilde{g}_i^x}(i_j)=I_{g_i^x}(i_j)$. Since
$1-t_0>1/2$, we can apply to $\tilde{g}_i^x$ the argument
presented above and get
\[
\sum_{j} I_{g_i^x}(i_j) = \sum_{j} I_{\tilde{g}_i^x}(i_j) \leq 6
Ent(1-t_0)=6 Ent(t_0).
\]
Thus, if $h(t) \geq Ent(t)$ for all $t$, then for all $x \in
[0,1]^n$,
\[
\sum_{j} I_{g_i^x}(i_j) \leq 6 I_{f_i^x}^h(i),
\]
as asserted.
\end{proof}

\begin{remark}
We note that $\overline{I}_f(i)=\sum_j I_g(i_j)$ (where $g$ is as
defined above) was treated in~\cite{BKKKL} as an alternative
definition of influences in the continuous case. It follows from
the proof of Proposition~\ref{Prop:Discretization} that this
definition is equivalent, up to a multiplicative constant, to
$h$-influence with $h(t)=Ent(t)$. As follows from
Theorem~\ref{Our-Main-Theorem}, this is the ``optimal'' definition
of influence for which the BKKKL theorem holds.
\end{remark}

\section{Generalized BKKKL Theorem and Applications}
\label{sec:applications}

In this section use the monotonization and discretization techniques
developed in Section~\ref{sec:lemmas} to prove a tight version of the 
BKKKL theorem for $h$-influences (Theorem~\ref{Our-Main-Theorem}), and to
generalize several known results concerning influences in
product spaces. The results we generalize include:
\begin{itemize}
\item A lower bound on the vector of influences obtained
by Talagrand~\cite{Talagrand1} (Section~\ref{sec:sub:Talagrand}),

\item A characterization of functions with a low sum of influences,
obtained by Dinur, Friedgut, and Hatami~\cite{Friedgut1,Friedgut2,Hatami}
(Section~\ref{sec:sub:Hatami}),

\item A relation between the measure of the boundary of a subset of 
the discrete cube and its influences, obtained by Margulis and 
Talagrand~\cite{Margulis,Talagrand3} (Section~\ref{sec:sub:boundary}), and

\item A lower bound on the correlation between monotone subsets of the
continuous cube in the average case, obtained in~\cite{Keller} 
(Section~\ref{sec:sub:Kleitman-Average}).
\end{itemize}  

\subsection{BKKKL Theorem for $h$-Influences: Proof of
Theorem~\ref{Our-Main-Theorem}}
\label{sec:sub:BKKKL}

In the proof of the theorem, we follow the simple proof
of the BKKKL theorem presented in~\cite{Friedgut2}. The proof uses
the following generalization of the KKL theorem presented
in~\cite{Friedgut-Kalai}:
\begin{theorem}[Friedgut and Kalai]
There exists a constant $c>0$ such that the following holds:
Consider the discrete cube $\{0,1\}^n$ endowed with the uniform
measure. Let $f:\{0,1\}^n \rightarrow \{0,1\}$, such that
$\mathbb{E}f=p$. If for all $k$, $I_f(k) \leq \delta$, then
\[
\sum_k I_f(k) \geq c p(1-p) \log(1/\delta).
\]
\label{Thm:Generalized-KKL}
\end{theorem}
The proof of Theorem~\ref{Our-Main-Theorem} is a
straightforward combination of Theorem~\ref{Thm:Monotonization},
Proposition~\ref{Prop:Discretization}, and
Theorem~\ref{Thm:Generalized-KKL}, as follows:

\medskip

\noindent \textbf{Proof of Theorem~\ref{Our-Main-Theorem}:} Let
$f:[0,1]^n \rightarrow \{0,1\}$, such that $\mathbb{E}f=p$. We
want to find a lower bound on the $h$-influences of $f$. By
Theorem~\ref{Thm:Monotonization}, we can replace $f$ by a monotone
function with the same expectation and lower 
$h$-influences. A lower bound on the $h$-influences of the new
function implies the same bound on the $h$-influences on $f$, and
hence we can assume w.l.o.g. that $f$ is monotone. Since $h(t)
\geq Ent(t)$ for all $t$, we can use
Proposition~\ref{Prop:Discretization} to approximate $f$ by a
function $g:\{0,1\}^{nl} \rightarrow \{0,1\}$, such that for all
$i$, $\sum_j I_g(i_j) \leq 6 I^h_f(i)$. Now assume that all the
$h$-influences of $f$ satisfy $I^h_f(i) \leq c_1 \log n/n$ for
some constant $c_1$. In this case, for all $i,j$ we have $I_g(i_j)
\leq 6 c_1 \log n/n$, and thus, by
Theorem~\ref{Thm:Generalized-KKL},
\[
\sum_{i} I^h_f(i) \geq \frac{1}{6} \sum_{i,j} I_g(i_j) \geq c_2
p(1-p) \log n,
\]
for some constant $c_2$. Therefore, there exists a coordinate $i$
such that $I^h_f(i) \geq c_2 p(1-p) \log n/n$, as asserted.
$\blacksquare$

\bigskip

\noindent A slight modification of the proof yields the following
generalization of Theorem~\ref{Thm:Generalized-KKL} to $h$-influences,
that shall be used later:
\begin{proposition}
Let $h:[0,1] \rightarrow [0 ,1]$ be a concave function
satisfying $h(t) \geq Ent(t)$ for all $0 \leq t \leq 1$. There
exists a constant $c>0$ such that the following holds: Let
$f:[0,1]^n \rightarrow \{0,1\}$, such that $\mathbb{E}f=p$. If for
all $k$, $I^h_f(k) \leq \delta$, then
\[
\sum_k I^h_f(k) \geq c p(1-p) \log(1/\delta).
\]
\label{Prop:Generalized-KKL-Our}
\end{proposition}

\begin{remark} 
We note that the strongest form of the assertion of Theorem~\ref{Our-Main-Theorem}
is obtained for the function $h(t)=Ent(t)$. The same holds also for 
Proposition~\ref{Prop:Generalized-KKL-Our}, and for Propositions~
\ref{Prop:Our-Talagrand-General}, \ref{Prop:Continuous-Junta}, 
\ref{Prop:Hatami}, and \ref{Prop:Boundary}, presented in the following
sections. However, for the sake of generality, we state these propositions for
a general $h$-influence, satisfying the condition ``$h(t) \geq Ent(t)$ for all 
$0 \leq t \leq 1$''.
\end{remark} 

\subsection{Generalization of Talagrand's Lower Bound on the
Vector of Influences}
\label{sec:sub:Talagrand}

In~\cite{Talagrand1}, Talagrand proved the following strengthening
of the KKL theorem for functions on the discrete cube endowed with
a general product measure:
\begin{theorem}[Talagrand]
Consider the discrete cube $\{0,1\}^n$ endowed with the product
measure $\mu_q$ defined by $\mu_q(x)=q^{\sum x_i} (1-q)^{n- \sum
x_i}$. There exists a constant $K>0$ such that for any function
$f:\{0,1\}^n \rightarrow \{0,1\}$ with $\mathbb{E}f=p$,
\begin{equation}
p(1-p) \leq K q(1-q) \log \frac{2}{q(1-q)} \sum_{i \leq n}
\frac{I_f(i)}{\log\frac{1}{q(1-q)I_f(i)}}.
\end{equation}
\label{Thm:Talagrand-General-q}
\end{theorem}

\noindent The proof of Theorem~\ref{Thm:Talagrand-General-q} uses
the biased Fourier-Walsh expansion for functions on the discrete
cube, and a biased version of Beckner's hypercontractive
inequality~\cite{Beckner}, which is proved
in~\cite{Talagrand1}.\footnote{The biased hypercontractive
inequality proved in~(\cite{Talagrand1}, Lemma~2.1) yields a
hypercontractivity constant of $q(1-q)$. As was shown later
in~\cite{Oles}, the optimal constant is bigger (of order
$Ent(q)$). However, this in-optimality affects the assertion of
Theorem~\ref{Thm:Talagrand-General-q} only in the constant factor
$K$ which is not specified in~\cite{Talagrand1}.}

\medskip

\noindent Using the technique presented above, we can generalize
Theorem~\ref{Thm:Talagrand-General-q} to $h$-influences:
\footnote{Another generalization of Theorem~\ref{Thm:Talagrand-General-q}
to the continuous case was obtained recently by Hatami~\cite{Hatami},
for a combination of Definitions~\ref{Def:BKKKL} and~\ref{Def:Tilde} of
the influences. Hatami's result does not follow from 
Proposition~\ref{Prop:Our-Talagrand-General}, but also does not imply it.}
\begin{proposition}
Let $h:[0,1] \rightarrow [0,1]$ be a concave function
satisfying $h(t) \geq Ent(t)$ for all $0 \leq t \leq 1$. There
exists a constant $K>0$ such that for any function $f:[0,1]^n
\rightarrow \{0,1\}$ with $\mathbb{E}f=p$,
\begin{equation}
p(1-p) \leq K \sum_{i \leq n}
\frac{I^h_f(i)}{\log\frac{4}{3I^h_f(i)}}.
\label{Eq:Our-Talagrand-General}
\end{equation}
\label{Prop:Our-Talagrand-General}
\end{proposition}
\begin{proof}
Let $f:[0,1]^n \rightarrow \{0,1\}$. As in the proof of
Theorem~\ref{Our-Main-Theorem}, monotonization and discretization
allow to replace $f$ by a function $g:\{0,1\}^{nl} \rightarrow
\{0,1\}$, such that $\mathbb{E}g=\mathbb{E}f=p$, and for all $i$,
$\sum_{j} I_g(i_j) \leq 6 I^h_f(i)$. Applying
Theorem~\ref{Thm:Talagrand-General-q} in the case $q=1/2$ to $g$,
we get
\begin{equation}
p(1-p) \leq K \sum_{1 \leq i \leq n, 1 \leq j \leq l}
\frac{I_g(i_j)}{\log\frac{4}{I_g(i_j)}} \leq 2K \sum_{1 \leq i
\leq n, 1 \leq j \leq l} \frac{I_g(i_j)}{\log\frac{8}{I_g(i_j)}}.
\label{Eq:Talagrand-General-Our}
\end{equation}
Note that for all $1\leq i \leq n$,
\[
\sum_{1 \leq j \leq l} \frac{I_g(i_j)}{\log\frac{8}{I_g(i_j)}}
\leq \sum_{1 \leq j \leq l} \frac{I_g(i_j)}{\log\frac{8}{\sum_{1
\leq j \leq l} I_g(i_j)}} = \frac{\sum_{1 \leq j \leq l}
I_g(i_j)}{\log\frac{8}{\sum_{1 \leq j \leq l} I_g(i_j)}}.
\]
Since the function $\varphi(x)=x/\log(8/x)$ is monotone increasing
in $x$ in $[0,8)$, 
we get
\[
\sum_{1\leq i\leq n} \sum_{1 \leq j \leq l}
\frac{I_g(i_j)}{\log\frac{8}{I_g(i_j)}} \leq \sum_{1 \leq i \leq
n} \frac{\sum_{1 \leq j \leq l} I_g(i_j)}{\log\frac{8}{\sum_{1
\leq j \leq l} I_g(i_j)}} \leq \sum_{1 \leq i \leq n} \frac{ 6
I_f^h(i)}{\log\frac{8}{6 I_f^h(i)}}.
\]
Finally, substitution into
Inequality~(\ref{Eq:Talagrand-General-Our}) yields:
\[
p(1-p) \leq 2K \sum_{1\leq i\leq n} \sum_{1 \leq j \leq l}
\frac{I_g(i_j)}{\log\frac{8}{I_g(i_j)}} \leq 2K \sum_{1 \leq i
\leq n} \frac{6 I_f^h(i)}{\log\frac{8}{6 I_f^h(i)}} = K'
\sum_{1 \leq i \leq n} \frac{I_f^h(i)}{\log\frac{4}{3 I_f^h(i)}} ,
\]
as asserted.
\end{proof}

\noindent We note that Theorem~\ref{Thm:Talagrand-General-q} (for
a general $q$) follows from
Proposition~\ref{Prop:Our-Talagrand-General}, using a standard
transformation from the biased measure on the discrete cube to the
Lebesgue measure on the continuous cube. Indeed, consider the
discrete cube $\{0,1\}^n$ endowed with the measure $\mu_q$. Define
$G:[0,1] \rightarrow \{0,1\}$ by $G(x)=0$ if $x \leq 1-q$, and
$G(x)=1$ if $x>1-q$. For a function $f:\{0,1\}^n \rightarrow
\{0,1\}$, define $\tilde{f}:[0,1]^n \rightarrow \{0,1\}$ by
\[
\tilde{f}(x_1,x_2,\ldots,x_n)=f(G(x_1),G(x_2),\ldots,G(x_n)).
\]
It is easy to see that $\mathbb{E}\tilde{f}=\mathbb{E}f$, where
the expectation in the left hand side is w.r.t. the Lebesgue
measure on the continuous cube, and the expectation in the right
hand side is w.r.t. the measure $\mu_q$ on the discrete cube.
Applying Proposition~\ref{Prop:Our-Talagrand-General} to the
function $\tilde{f}$, we get
\begin{equation}
p(1-p) \leq K \sum_{1 \leq i \leq n}
\frac{I_{\tilde{f}}^h(i)}{\log\frac{4}{3 I_{\tilde{f}}^h(i)}},
\label{Eq:Talagrand-General-Our-2}
\end{equation}
where $p=\mathbb{E}\tilde{f}=\mathbb{E}f$. Due to the construction
of $\tilde{f}$, the contribution of each non-constant fiber to an
influence of $\tilde{f}$ is $h(q)$, and hence taking
$h(t)=Ent(t)$, the contribution of each non-constant fiber is
$Ent(q)$. Thus, for all $1 \leq i \leq n$, we have:
\[
I_{\tilde{f}}^h(i) = Ent(q) I_f(i),
\]
where the influence in the right hand side is w.r.t. the measure
$\mu_q$ on the discrete cube. Substituting into
Inequality~(\ref{Eq:Talagrand-General-Our-2}), we get:
\begin{equation}
p(1-p) \leq K Ent(q) \sum_{1 \leq i \leq n}
\frac{I_f(i)}{\log\frac{4}{3 Ent(q) I_f(i)}}.
\label{Eq:Talagrand-General-Our-3}
\end{equation}
Finally, there exist constants $K_1$ and $K_2$ such that for all
$q$ and all $I_f(i)$,
\[
Ent(q) \leq K_1 q(1-q)\log\frac{2}{q(1-q)},
\]
and
\[
\log\frac{4}{3 Ent(q) I_f(i)} \geq K_2 \log\frac{1}{q(1-q)
I_f(i)}.
\]
Therefore, the assertion of Theorem~\ref{Thm:Talagrand-General-q}
follows from Inequality~(\ref{Eq:Talagrand-General-Our-3}).

\medskip

\noindent This shows that the biased hypercontractive inequality
used in the proof in~\cite{Talagrand1} can be replaced by the
original hypercontractive inequality (i.e. the inequality for the
uniform measure), combined with our argument presented above.

\subsection{Functions with a Low Sum of Influences}
\label{sec:sub:Hatami}

One of the most useful results concerning influences of variables
on functions on the discrete cube is the following theorem, due to
Friedgut~\cite{Friedgut1}, asserting that if the sum of influences
is small then the function essentially depends on a few
coordinates.
\begin{theorem}[Friedgut]
Consider $\{0,1\}^n$ as a measure space with the uniform measure.
Let $f:\{0,1\}^n \rightarrow \{0,1\}$ such that $\sum_k I_f(k)=t$,
and let $\epsilon>0$. Denote $M=t/\epsilon$. There exists a
Boolean function $g$ depending only on
$\exp((2+\sqrt{\frac{2\log(4M)}{M}})M)$ variables, such that
$\Pr[f \neq g] \leq \epsilon$.
\label{Thm:Junta}
\end{theorem}

\noindent Dinur and Friedgut~\cite{Friedgut2} observed that using
discretization, Theorem~\ref{Thm:Junta} can be generalized to
monotone functions on the continuous cube:
\begin{theorem}[Dinur and Friedgut]
There exists a constant $c>0$ such that the following holds: Let
$\epsilon>0$, and let $f:[0,1]^n \rightarrow \{0,1\}$ be a
monotone function. If $\sum_i I_f(i) \leq B$, then there exists a
set $J \subset \{1,2,\ldots,n\}$ with $|J|\leq \exp(cB/\epsilon)$
and a function $g:[0,1]^n \rightarrow \{0,1\}$ depending only on
the coordinates in $J$, such that $||f-g||_2^2 \leq \epsilon$.
\label{Thm:Continuous-Junta}
\end{theorem}
Dinur and Friedgut~\cite{Friedgut2} conjectured that the assertion
of Theorem~\ref{Thm:Continuous-Junta} holds even without the
monotonicity assumption. This conjecture was disproved by
Hatami~\cite{Hatami}. On the other hand, Hatami proved (for
general Boolean functions on the continuous cube) that if the sum
of influences of $f$ is small, then $f$ can be approximated by a
function having a {\it decision tree of bounded depth}
(see~\cite{Hatami} for the definitions).
\begin{theorem}[Hatami]
Let $f:[0,1]^n \rightarrow \{0,1\}$ satisfy $\sum_i I_f(i) \leq
B$. Then for every $\epsilon>0$, there exists a function
$g:[0,1]^n \rightarrow \{0,1\}$ such that $||f-g||_2^2 \leq
\epsilon$, and $g$ has a decision tree of depth at most
$\exp(cB/\epsilon^2)$, where $c$ is a universal constant.
\label{Thm:Hatami}
\end{theorem}

\medskip

\noindent Using the techniques presented above we can generalize
Theorems~\ref{Thm:Continuous-Junta} and~\ref{Thm:Hatami} to
$h$-influences.
\begin{proposition}
Let $h:[0,1] \rightarrow [0,1]$ satisfy $h(t) \geq Ent(t)$ for all
$0 \leq t \leq 1$. There exists a constant $c>0$ such that the
following holds: Let $\epsilon>0$, and let $f:[0,1]^n \rightarrow
\{0,1\}$ be a monotone function. If $\sum_i I_f^h(i) \leq B$, then
there exists a set $J \subset \{1,2,\ldots,n\}$ with $|J|\leq
\exp(cB/\epsilon)$ and a function $g:[0,1]^n \rightarrow \{0,1\}$
depending only on the coordinates in $J$, such that $||f-g||_2^2
\leq \epsilon$.
\label{Prop:Continuous-Junta}
\end{proposition}
Proposition~\ref{Prop:Continuous-Junta} follows immediately from
Theorem~\ref{Thm:Junta} using
Proposition~\ref{Prop:Discretization}. As in
Section~\ref{sec:sub:Talagrand}, we can apply
Proposition~\ref{Prop:Continuous-Junta} with $h(t)=Ent(t)$ to
functions on the discrete cube endowed with the measure $\mu_q$ to
get:
\begin{proposition}
There exists an absolute constant $c>0$ such that the following
holds: Consider the discrete cube $\{0,1\}^n$ endowed with the
measure $\mu_q$. If a monotone function $f:\{0,1\}^n \rightarrow
\{0,1\}$ satisfies $\sum_i I_f(i) \leq B$, then there exists
$g:\{0,1\}^n \rightarrow \{0,1\}^n$ depending on at most
$\exp(cEnt(q)B/\epsilon)$ coordinates, such that $||f-g||_2^2 \leq
\epsilon$.
\label{Prop:q-Junta}
\end{proposition}
For monotone functions, Proposition~\ref{Prop:q-Junta} gives a more 
precise result than Theorem~4.1 in~\cite{Friedgut1} (in which the exact
dependence on $q$ was not specified).

\medskip

\noindent The generalization of Theorem~\ref{Thm:Hatami} is as
follows:
\begin{proposition}
Let $h:[0,1] \rightarrow [0,1]$ be a concave function
satisfying $h(t) \geq Ent(t)$ for all $0 \leq t \leq 1$. Let
$f:[0,1]^n \rightarrow \{0,1\}$ satisfy $\sum_i I_f^h(i) \leq B$.
Then for every $\epsilon>0$, there exists a function $g:[0,1]^n
\rightarrow \{0,1\}$ such that $||f-g||_2^2 \leq \epsilon$, and
$g$ has a decision tree of depth at most $\exp(cB/\epsilon^2)$,
where $c$ is a universal constant. \label{Prop:Hatami}
\end{proposition}
The proof of Proposition~\ref{Prop:Hatami} is a minor modification
of the proof of Theorem~\ref{Thm:Hatami} presented
in~\cite{Hatami}. The only two changes are replacing ordinary
influences with $h$-influences throughout the proof, and replacing
``Theorem~B'' used in the proof by the following proposition,
which follows immediately from
Proposition~\ref{Prop:Generalized-KKL-Our}.
\begin{proposition}
Let $h:[0,1] \rightarrow [0,1]$ be a concave function
satisfying $h(t) \geq Ent(t)$ for all $0 \leq t \leq 1$. For all
$f:[0,1]^n \rightarrow \{0,1\}$ with $\mathbb{E}f=p$, there exists
$1 \leq i \leq n$, such that
\[
I_f^h(i) \geq e^{-\frac{c}{p(1-p)}\sum_{j=1}^n I_f^h(j)},
\]
where $c$ is a universal constant.
\end{proposition}

\medskip

\noindent In Section~\ref{sec:Tightness} we show that for
some cases of interest, the advantage of
Proposition~\ref{Prop:Continuous-Junta} over
Theorem~\ref{Thm:Continuous-Junta} is significant.

\subsection{A Relation Between the Sum of Influences and the Size
of the Boundary}
\label{sec:sub:boundary}

\begin{definition}
Consider a product space $X=X_1 \times X_2 \times \ldots \times X_n$ 
endowed with a product measure $\mu=\mu_1 \otimes \ldots \otimes \mu_n$, 
and let $A \subset X$. The boundary of $A$ is
\[
\partial A = \{x \in A: \exists (1 \leq i \leq n), \mbox{ s.t. the
function } 1_A \mbox{ is non-constant on the fiber } s_i(x) \}.
\]
\end{definition}
In~\cite{Margulis}, Margulis proved that for subsets of the
discrete cube endowed with the uniform measure, the size of the
boundary and the sum of influences cannot be small simultaneously.
For monotone subsets of the discrete cube,
Talagrand~\cite{Talagrand3} gave the following precise form to
this statement:
\begin{theorem}[Talagrand]
Consider the discrete cube $\{0,1\}^n$ endowed with the uniform
measure $\mu$. For any monotone subset $A \subset \{0,1\}^n$ with
$\mu(A) \leq 1/2$,
\[
\mu(\partial A) \sum_i I_{1_A}(i) \geq c \mu(A)^2 \log
\frac{e}{\mu(A)},
\]
where $c$ is a universal constant.
\label{Thm:Boundary}
\end{theorem}

\noindent Theorem~\ref{Thm:Boundary} can be generalized to subsets
of the continuous cube, as follows:
\begin{proposition}
Let $h:[0,1] \rightarrow [0,1]$ be a function satisfying $h(t)
\geq Ent(t)$ for all $0 \leq t \leq 1$. Consider the continuous
cube $[0,1]^n$ endowed with the Lebesgue measure $\lambda$. For
any monotone subset $A \subset [0,1]^n$ with $\lambda(A) \leq
1/2$,
\begin{equation}
\lambda(\partial A) \sum_i I_{1_A}^h(i) \geq c \lambda(A)^2 \log
\frac{e}{\lambda(A)},
\label{Eq:Boundary}
\end{equation}
where $c$ is a universal constant.
\label{Prop:Boundary}
\end{proposition}

\begin{proof}
As in the proof of Theorem~\ref{Our-Main-Theorem}, discretization
allows to replace the function $1_A$ by a function $g:\{0,1\}^{nl}
\rightarrow \{0,1\}$, such that $\mathbb{E}g=\mathbb{E}1_A$, and
for all $i$, $\sum_{j} I_g(i_j) \leq 6I^h_{1_A}(i)$. Denote the
subset of the discrete cube corresponding to $g$ by $B$. Applying
Theorem~\ref{Thm:Boundary} to $B$, we get
\begin{equation}
\mu(\partial B) \sum_{i,j} I_{g}(i_j) \geq c \mu(B)^2 \log
\frac{e}{\mu(B)}. 
\label{Eq:Boundary1}
\end{equation}
By the construction, we have $\sum_{i,j} I_{g}(i_j) \leq 6 \sum_i
I^h_{1_A}(i)$, and $\mu(B)=\lambda(A)$. In addition, we observe
that $\mu(\partial B) \leq \lambda(\partial A)$. Indeed, if $x \in
\partial B$ then there exist $i,j$ such that either ($x \in B$ and
$x \oplus e_{i_j} \not \in B$) or ($x \not \in B$ and $x \oplus
e_{i_j} \in B$). Recall that the pre-image of $x$ (in the
discretization procedure) is a sub-cube of the continuous cube,
denoted by $A_x$, such that $\lambda(A_x) = 2^{-nl}$. For all $y
\in A_x$, the function $1_A$ is non-constant on the fiber
$s_i(y)$, and hence $A_x \subset
\partial A$. Hence, $\mu(\partial B) \leq \lambda(\partial A)$.
The assertion of Proposition~\ref{Prop:Boundary} (with the
constant $c/6$) follows now immediately from
Inequality~(\ref{Eq:Boundary1}).
\end{proof}

\noindent We note that a stronger version of
Theorem~\ref{Thm:Boundary} was proved by Talagrand
in~\cite{Talagrand4}. It seems challenging to find a
generalization of this version to the continuous cube.

\subsection{Lower Bound on the Correlation Between Monotone
Subsets of the Continuous Cube}
\label{sec:sub:Kleitman-Average}

One of the most well-known correlation inequalities for monotone
subsets of the discrete cube is the following inequality, due to
Harris~\cite{Harris} and Kleitman~\cite{Kleitman}:
\begin{theorem}[Harris, Kleitman]
Let $A,B$ be monotone subsets of $\{0,1\}^n$ endowed with the
uniform measure $\mu$. Then
\begin{equation}
\mu(A \cap B) \geq \mu(A) \mu(B),
\end{equation}
i.e., the correlation of $A$ and $B$ is nonnegative.
\label{Kleitman}
\end{theorem}
In~\cite{Talagrand2}, Talagrand established a lower bound on the
correlation in terms of how much the two sets depend
simultaneously on the same coordinates.
\begin{theorem}[Talagrand]
Let $A,B$ be monotone subsets of $\{0,1\}^n$ endowed with the
uniform measure $\mu$. Then
\begin{equation}
\mu(A \cap B) - \mu(A) \mu(B) \geq K \varphi(\sum_{i \leq n}
I_{1_A}(i) I_{1_B}(i)),
\label{Eq:Talagrand-Correlation}
\end{equation}
where $\varphi(x)=x/\log(e/x)$ and $K$ is a universal constant.
\label{Thm:Talagrand-Correlation}
\end{theorem}
In~\cite{Keller} it was shown that when the correlation is
averaged amongst all $A,B \in T$ for any family $T$ of monotone
subsets of $\{0,1\}^n$, the lower bound asserted in
Theorem~\ref{Thm:Talagrand-Correlation} can be improved.
\begin{proposition}[\cite{Keller}]
Let $T$ be a family of monotone subsets of the discrete cube
endowed with the uniform measure $\mu$. Then
\begin{equation}
\sum_{A,B \in T} (\mu(A \cap B) - \mu(A) \mu(B)) \geq \frac{1}{4}
\sum_{A,B \in T} \sum_{i \leq n} I_{1_A}(i) I_{1_B}(i).
\label{Eq:Our-Theorem-Correlation}
\end{equation}
\label{Prop:Our-Theorem-Correlation}
\end{proposition}
Furthermore, it was shown in~\cite{Keller} that
Proposition~\ref{Prop:Our-Theorem-Correlation} can be generalized
to the correlation of monotone subsets of the continuous cube,
using Definition~\ref{Def:Tilde} of the influences in the
continuous case.
\begin{proposition}[\cite{Keller}]
Let $T$ be a family of monotone subsets of the continuous cube
endowed with the Lebesgue measure $\lambda$. Then
\begin{equation}
\sum_{A,B \in T} (\lambda(A \cap B) - \lambda(A) \lambda(B)) \geq
3 \sum_{A,B \in T} \sum_{i \leq n} \tilde{I}_{1_A}(i)
\tilde{I}_{1_B}(i).
\label{Eq:Our-Theorem-Correlation-Cont}
\end{equation}
\label{Prop:Our-Theorem-Correlation-Cont}
\end{proposition}
It was also shown in~\cite{Keller} that the natural generalization
of Proposition~\ref{Prop:Our-Theorem-Correlation} does not hold
under Definition~\ref{Def:BKKKL} of the influences in the
continuous case.

\medskip

\noindent Using $h$-influences, the assertion of
Proposition~\ref{Prop:Our-Theorem-Correlation-Cont} can be
generalized significantly:
\begin{proposition}
Let $1/2 < \alpha \leq 1$, and let $h_{\alpha}:[0,1] \rightarrow
[0,1]$ be defined by $h_{\alpha}(t)=[t(1-t)]^{\alpha}$. There
exists $c=c(\alpha)$ such that for any family $T$ of monotone
subsets of the continuous cube endowed with the Lebesgue measure
$\lambda$,
\begin{equation}
\sum_{A,B \in T} (\lambda(A \cap B) - \lambda(A) \lambda(B)) \geq
c(\alpha) \sum_{A,B \in T} \sum_{i \leq n} I^{h_{\alpha}}_{1_A}(i)
I^{h_{\alpha}}_{1_B}(i).
\label{Eq:Correlation-New}
\end{equation}
\label{Prop:Correlation-New}
\end{proposition}
\begin{proof}
The proof is similar to the proof of
Proposition~\ref{Prop:Our-Theorem-Correlation-Cont}, presented
in~\cite{Keller}. The main observation in the proof is that for
$1/2 < \alpha \leq 1$, the influences $I^{h_{\alpha}}_{1_A}(i)$
can be represented as the Fourier coefficients of the function
$1_A$ with respect to an appropriate orthonormal system of
functions in $L^2([0,1])$. Consider the functions $r'_i:(0,1)
\rightarrow \mathbb{R}$ defined by:
\[
r'_i(t)=\alpha(2t-1)[t(1-t)]^{\alpha-1}.
\]
Since $\int_0^1 r'_i(t)dt=0$, the functions $\{r'_i\}_{i=1}^n$
along with the constant function constitute an orthogonal system
in $L^2([0,1])$. Let
\[
c(\alpha)=\Big(\int_0^1 (\alpha(2t-1)[t(1-t)]^{\alpha-1})^2 dt
\Big)^{-1}.
\]
Note that the integral is finite since $\alpha>1/2$. The functions
$r_i(t)=c(\alpha)^{1/2} r'_i(t)$ satisfy $\int_0^1 r_i^2(t)dt =
1$, and hence the system $\{r_i\}_{i=1}^n$ along with the constant
function $r_{\emptyset} \equiv 1$ is an orthonormal system of
functions in $L^2([0,1])$. By~(\cite{Keller}, Lemma~7), this
implies that for any family $T$ of monotone subsets of the
continuous cube,
\begin{equation}
\sum_{A,B \in T} (\lambda(A \cap B) - \lambda(A) \lambda(B)) \geq
\sum_{A,B \in T} \sum_{i \leq n} \int_{[0,1]^n} 1_A r_i
(x)d\lambda(x) \int_{[0,1]^n} 1_B r_i(x)d\lambda(x).
\label{Eq:Correlation-New-2}
\end{equation}
We observe that for all $1 \leq i \leq n$,
\[
\int_{[0,1]^n} 1_A r_i (x)d\lambda(x) = c(\alpha)^{1/2}
I^{h_{\alpha}}_{1_A},
\]
and similarly for $1_B$. Indeed, consider a single fiber $s_i(x)$,
and denote the restriction of the function $1_A$ to the fiber
$s_i(x)$ by $f_i^x:[0,1] \rightarrow \{0,1\}$. Since the function
$f_i^x$ is Boolean and monotone, there exists $t_0$ such that for
all $t<t_0$, $f_i^x(t)=0$, and for all $t>t_0$, $f_i^x(t)=1$. In
this case we have
\[
\int_0^1 f_i^x r_i(t)dt = \int_{t_0}^1 c(\alpha)^{1/2}
\alpha(2t-1)[t(1-t)]^{\alpha-1} dt = c(\alpha)^{1/2}
[t_0(1-t_0)]^{\alpha} = c(\alpha)^{1/2}
h_{\alpha}(\mathbb{E}f_i^x).
\]
Hence,
\[
\int_{[0,1]^n} 1_A r_i(x)d\lambda(x) = \mathbb{E}_{x \in [0,1]^n}
\int_0^1 f_i^x r_i(t)dt = c(\alpha)^{1/2} \mathbb{E}_{x \in
[0,1]^n} h_{\alpha}(\mathbb{E}f_i^x) = c(\alpha)^{1/2}
I^{h_{\alpha}}_{1_A}(i).
\]
Substituting into Inequality~(\ref{Eq:Correlation-New-2}), the
assertion of the proposition follows.
\end{proof}

\noindent The following example demonstrates that the natural
generalization of Proposition~\ref{Prop:Our-Theorem-Correlation}
does not hold for $h$-influences of the class
$h_{\alpha}(t)=[t(1-t)]^{\alpha}$ with $\alpha<1/2$.
\begin{example}
Let $A \subset [0,1]^n$ be defined by
\[
A=\{x \in [0,1]^n:(\forall i: x_i \geq 1/n)\}.
\]
Clearly, $\lambda(A)=(1-1/n)^n \approx 1/e$, and hence,
\[
\lambda(A \cap A) - \lambda(A) \lambda(A) \approx (1/e)-(1/e)^2 =
\Theta(1).
\]
The function $1_A$ is non-constant on a fiber $s_k(x)$ if and only
if $x_j \geq 1/n$ for all $j \neq k$. Hence, the
$h_{\alpha}$-influence of the $k$-th coordinate on $1_A$ is
\[
I^{h_{\alpha}}_{1_A}(k)=(1-1/n)^{n-1} (\frac{n-1}{n^2})^{\alpha} =
\Theta(n^{-\alpha}),
\]
and thus,
\[
\sum_k I^{h_{\alpha}}_{1_A}(k) I^{h_{\alpha}}_{1_A}(k) =
\Theta(n^{1-2\alpha}).
\]
Therefore, if $\alpha<1/2$ then the assertion of
Proposition~\ref{Prop:Correlation-New} is far from being correct,
even for the self-correlation of the set $A$.
\end{example}

\medskip

\noindent If the assertion of
Proposition~\ref{Prop:Correlation-New} holds for $\alpha=1/2$,
then the above example shows its tightness. However, we weren't
able to prove the assertion in this case. It seems challenging to
find a generalization of Talagrand's
Theorem~\ref{Thm:Talagrand-Correlation} to the continuous cube
using $h$-influences, but it seems that the techniques of the
current paper cannot provide such result.

\section{Tightness of Results}
\label{sec:Tightness}

The tightness of most of our results can be shown using a biased
variant of the {\it tribes} function presented in~\cite{BL}. We
note that this function was already used to show the tightness of
the results in~\cite{Friedgut1,Friedgut-Kalai,Talagrand1}. In the
construction below we assume that $q \leq 1/2$. The case $q>1/2$
is treated in the end of this section.

\subsection{Tightness of Theorem~\ref{Our-Main-Theorem}}

Consider the discrete cube $\{0,1\}^n$ endowed with the
product measure $\mu_q$. Partition the set $\{1,2,\ldots,n\}$ into
sets $\{T_1,T_2,\ldots,T_{n/r}\}$ of size
\[
r=\frac{\log n - \log \log n + \log \log (1/q)}{\log(1/q)}
\]
each, and define $f:\{0,1\}^n \rightarrow \{0,1\}$ by setting
$f(x)=1$ if and only if there exists $i$ such that $x_j=1$ for all
$j \in T_i$. This function can be transformed to a function on the
continuous cube as follows: Define $G:[0,1] \rightarrow \{0,1\}$
by $G(x)=0$ if $x \leq 1-q$, and $G(x)=1$ if $x>1-q$, and let
$\tilde{f}:[0,1]^n \rightarrow \{0,1\}$ be defined by
\[
\tilde{f}(x_1,x_2,\ldots,x_n)=f(G(x_1),G(x_2),\ldots,G(x_n)).
\]
It is easy to see that
\[
\mathbb{E}\tilde{f} = 1-(1-q^r)^{n/r} \approx 1-1/e,
\]
where the expectation is with respect to the Lebesgue measure on
the continuous cube. The function $\tilde{f}$ is non-constant on a
fiber $s_k(x)$ where $k \in T_i$ if for all $j \in T_i$ we have
$x_j>1-q$, and for each $i' \neq i$, there exists $j \in T_{i'}$
such that $x_j \leq 1-q$. In each of the non-constant fibers we
have $\mathbb{E}\tilde{f}_k^x = q$. Thus, for all $h:[0,1]
\rightarrow [0,1]$ and for all $1 \leq k \leq n$,
\[
I^h_{\tilde{f}}(k) = q^{r-1} (1-q^r)^{(n/r)-1} h(q) \approx
\frac{\log n}{(en) q \log (1/q)} h(q) \leq \frac{2\log n}{(en)
Ent(q)} h(q),
\]
where the last inequality holds since for all $q \leq 1/2$ we have
$q \log (1/q) \geq (1-q) \log (1/(1-q))$. This example shows that
the assertion of Theorem~\ref{Our-Main-Theorem} is tight:
\begin{proposition}
Let $h:[0,1] \rightarrow [0,1]$ and let $\epsilon>0$. If there
exists $0<q<1$ such that $h(q) \leq \epsilon Ent(q)$, then there
exists a function $f:[0,1]^n \rightarrow \{0,1\}$ such that
$\mathbb{E}f=\Theta(1)$, and for all $1\leq k \leq n$, the
h-influence of the $k$-th variable on $f$ satisfies $I_f^h(k) \leq
c \epsilon \log n/n$, where $c$ is a universal constant.
\label{Prop:Our-Tightness}
\end{proposition}

\subsection{Tightness of Proposition~\ref{Prop:Our-Talagrand-General}}

The same example shows the tightness of
Proposition~\ref{Prop:Our-Talagrand-General}. Indeed, if there
exists $q$ such that $h(q) \leq \epsilon Ent(q)$, then for the
corresponding function $\tilde{f}$, we have
\[
I^h_{\tilde{f}}(k) \leq \frac{2 \epsilon \log n}{en},
\]
for all $1 \leq k \leq n$. Since the function $g(x)=x/\log(4/3x)$
is monotone increasing in $x$, we have
\[
\sum_{k=1}^n \frac{I^h_{\tilde{f}}(k)}{\log\frac{4}{3I^h_{\tilde{f}}(k)}} 
= \sum_{k=1}^n g(I^h_{\tilde{f}}(k)) \leq n g
\Big(\frac{2 \epsilon \log n}{en} \Big) \approx c \frac{\epsilon
\log n}{\log (n/\epsilon)} \leq c' \epsilon,
\]
where the last inequality holds for $n \gg 1/\epsilon$. This
contradicts Inequality~(\ref{Eq:Our-Talagrand-General}) since for
$\tilde{f}$, the left hand side of the inequality is $\Theta(1)$.

\subsection{Tightness of Proposition~\ref{Prop:Continuous-Junta}}

In order to show the advantage of
Proposition~\ref{Prop:Continuous-Junta} over
Theorem~\ref{Thm:Continuous-Junta}, we consider a slight
modification of the tribes function examined above. For $0<q \leq
1/2$ and for $m<n$, we construct the tribes function
$\tilde{f}:[0,1]^m \rightarrow \{0,1\}$, as above. Then we define
$\overline{f}:[0,1]^n \rightarrow \{0,1\}$ by
$\overline{f}(x_1,x_2,\ldots,x_n)=\tilde{f}(x_1,x_2,\ldots,x_m)$.
Clearly, $\overline{f}$ depends on $m$ variables. By the
calculation presented above, we have
\[
\sum_k I^h_{\overline{f}}(k) = \sum_k I^h_{\tilde{f}}(k) \leq
\frac{2\log m}{e Ent(q)} h(q),
\]
and hence, for $h(t)=Ent(t)$, we get $\sum_k I^h_{\overline{f}}(k)
\leq 2\log m /e$. Therefore,
Proposition~\ref{Prop:Continuous-Junta} implies that
$\overline{f}$ can be approximated by a function depending on at
most $\exp(c\log m/\epsilon)$ variables, which is tight, up to the
$(c/\epsilon)$ factor. For comparison, note that $\sum_k
I_{\overline{f}}(k) \approx m/e$, and hence
Theorem~\ref{Thm:Continuous-Junta} implies approximation by a
function depending on $\exp(cm/\epsilon)$ variables, which is far
from being tight.

\subsection{Tightness of Proposition~\ref{Prop:Boundary}}

Finally, in order to show the tightness of
Proposition~\ref{Prop:Boundary}, we consider a balanced threshold
function on the biased discrete cube. Consider the discrete cube
$\{0,1\}^n$ endowed with the measure $\mu_q$, and let
\[
A=\{x \in \{0,1\}^n : \sum_i x_i > \lfloor nq \rfloor\}.
\]
Using the function $G$ defined above, $A$ can be transformed to
$\tilde{A} \subset [0,1]^n$. It is well-known that
$\lambda(\tilde{A})=\mu_q(A) = \Theta(1)$. We have
\[
\partial A = \{x \in \{0,1\}^n : \lfloor nq \rfloor\ \leq
\sum_i x_i \leq \lfloor nq \rfloor +1\},
\]
and hence it can be shown using Stirling's formula that
\[
\lambda(\partial \tilde{A}) = \mu_q(\partial A) \approx
\frac{2}{\sqrt{2 \pi n q(1-q)}}.
\]
The function $1_{\tilde{A}}$ can be non-constant on the fiber $s_k(x)$
only if $x \in \partial \tilde{A}$, and the expectation of
the function on each non-constant fiber is $q$. Thus,
\[
\sum_{i=1}^n I^h_{1_{\tilde{A}}}(i) \leq n \mu_q(\partial A) h(q)
\approx \frac{2n}{\sqrt{2 \pi n q(1-q)}} h(q).
\]
Taking $h(q)=Ent(q)$, we get
\[
\lambda(\partial \tilde{A}) \sum_i I_{1_{\tilde{A}}}^h(i) \leq
\frac{2}{\sqrt{2 \pi n q(1-q)}} \frac{2n}{\sqrt{2 \pi n q(1-q)}}
Ent(q) = \frac{4Ent(q)}{2\pi q(1-q)} \leq \frac{8\log(1/q)}{\pi}.
\]
Since for $\tilde{A}$, the right hand side of
Inequality~(\ref{Eq:Boundary}) is $\Theta(1)$, this implies that
Proposition~\ref{Prop:Boundary} is tight, up to a $(\log(1/q))$
multiplicative factor.

\subsection{The case $q>1/2$}

In order to show the tightness for $q>1/2$, we use the
{\it dual function}.
\begin{definition}
For a Boolean function $f:\{0,1\}^n \rightarrow \{0,1\}$, the dual
function $\overline{f}:\{0,1\}^n \rightarrow \{0,1\}$ is defined
by:
\[
\overline{f}(x_1,x_2,\ldots,x_n) = 1-f(1-x_1,1-x_2,\ldots,1-x_n).
\]
\end{definition}
It is easy to see that for all $0<q<1$,
\[
\mathbb{E}_q f = 1- \mathbb{E}_{1-q} \overline{f}
\]
where $\mathbb{E}_q(\cdot)$ denotes expectation w.r.t. the measure
$\mu_q$. Similarly, it is easy to show that for all $1 \leq i \leq
n$,
\[
I^q_f(i)=I^{1-q}_{\overline{f}}(i),
\]
where $I^q_f(i)$ denotes influence w.r.t. the measure $\mu_q$.
Therefore, the dual functions of the functions described above
show the tightness of our results for $q>1/2$.




\end{document}